\documentclass[10pt]{amsart}
\usepackage{amssymb}
\usepackage{bm}
\usepackage{graphicx}
\usepackage[centertags]{amsmath}
\usepackage{amsfonts}
\usepackage{amsthm}
\usepackage[usenames,dvipsnames]{xcolor}

\newtheorem*{sb157}{Scottish Book Problem 157}
\newtheorem*{theoremb}{Bruckner's Theorem 3.1}
\newtheorem*{theorembm}{Banach-Mazur Theorem}
\newtheorem*{theoremom}{O'Malley's Theorem 1$^*$}
\newtheorem*{theoremrolle}{Rolle's Theorem AP}
\newtheorem*{theoremmvt}{The Mean Value Theorem AP}
\newtheorem*{lemmaom}{O'Malley's Lemma}
\newtheorem*{theoremorn}{Ornstein's Theorem}
\newtheorem{theorem}{Theorem}
\newtheorem{thm}{Theorem}
\newtheorem{cor}[thm]{Corollary}
\newtheorem{lem}[thm]{Lemma}

\theoremstyle{definition}

\newtheorem{rem}[thm]{Remark}

\newtheorem{case}{Case}[thm]

\newcommand{\nn}{\mathbb{N}}

\newcommand{\ee}{\varepsilon}

\newcommand{\ds}{\displaystyle}

\def\R{\mathbb R}
\def\N{\mathbb N}
\def\c{C[0,1]}
\def\sb157{{\bf Scottish Book 157}}

\begin{document}

\title{On Scottish Book Problem 157}
\author{Kevin Beanland, Paul Humke, Trevor Richards}

\address{Department of Mathematics, Washington and Lee University, Lexington, VA 24450.}
\email{beanlandk@wlu.edu}
\email{humkep@gmail.com}
\email{richardst@wlu.edu}

\thanks{}

\thanks{2010 \textit{Mathematics Subject Classification}. Primary: }
\maketitle


\begin{abstract}
This paper describes our hunt for the solver of Problem 157 in the Scottish Book, a problem originally posed by A.~J. (Gus) Ward in 1937. We first make the observation that a theorem of Richard O'Malley from 1975 yields an immediate positive solution. A further look at O'Malley's references revealed a 1970 paper by Donald Ornstein that we now believe contains the first solution of {\em SB 157}. We isolate the common elements in the machinery used by both Ornstein and O'Malley and discuss several consequences. We also examine an example function given by Ornstein.  There are some difficulties with this function but we provide a fix, and show moreover that functions of that kind are typical in the sense of the Baire category theorem.

\end{abstract}

\section{The Solution in Brief}
On March 23rd, 1937 A.J. Ward asked the following problem which is recorded as Problem 157 in the famous {\em Scottish Book}.\footnote[1]{The prize for the solution to this problem is lunch at the ``The Dorothy" in Cambridge which the authors now offer to purchase for Richard O'Malley and Donald Ornstein; transportation costs are another matter!}

\begin{sb157}
Suppose $f$ is approximately continuous and at each point $x_0$ and the quantity $$\limsup_{h \to 0^+} \frac{f(x_0+h)-f(x)}{h},$$ neglecting any set of $h$ which have zero density at $h=0$,
is positive. Is $f(x)$ monotone increasing?
\end{sb157}

We became aware of this problem only recently when one of us\footnote[2]{Humke} was discussing the upcoming new edition of the {\it Scottish Book} with its editor, Dan Mauldin. Problem 157 was one of the problems marked as unresolved. Upon returning to campus, the authors decided to have a look at {\bf Scottish Book Problem 157} and a natural approach to its resolution soon brought them to Richard O'Malley's paper on approximate maxima, \cite{OM}. In that paper, O'Malley proves \cite[Theorem 1]{OM} from which the following result is an immediate corollary.

\begin{theoremom}\label{mom}
Let $f:[0,1]\to\mathbb{R}$ be approximately continuous but not strictly increasing. Then $f$ attains an approximate maximum at some point $x_0\in [0,1)$.
\end{theoremom}

A follow up conversation with O'Malley then led us to the following theorem by Donald Ornstein, \cite{O}.

\begin{theoremorn}
Let $f(x)$ be a real--valued function of a real variable satisfying the following:
\begin{enumerate}
\item[(a)] $f(x)$ is approximately continuous,

\item[(b)] For each $x_0$, let $E$ be the set of $x$, such that $f(x)-f(x_0)\ge 0$. Then
\[
\limsup_{h\to 0^+}\lambda\left( E\cap (x_0,x_0+h)\right)/h\not=0.
\]
\end{enumerate}

Then $f$ is monotone increasing and continuous.
\end{theoremorn}

This is clearly the solution to \sb157, and as far as we see is the first. There is a certain commonality to the machinery used in the proofs of O'Malley and Ornstein, and we'll try to isolate that common thread in the next section. Our proofs (largely reformulating those of O'Malley and Ornstein) will provide the slightly stronger result that the function under consideration is in fact strictly increasing (rather than just monotonically increasing).  We'll also list some elementary consequences and state all the relevant background.

In Section~\ref{sect: Ornstein's example.}, we'll examine an example given in~\cite{O}.  We will show that this example needs amending and supply that amendment. Finally, we will also show in that section that functions satisfying the properties of Ornstein's example are typical in the sense of the Baire category theorem.

But for those familiar with the definitions, here is the solution to~\sb157 which we have drawn from O'Malley's work and which is a trivial consequence of Theorem~1$^*$.

\begin{thm}
Suppose that $f:[0,1]\to\mathbb{R}$ is approximately continuous and that, for each $x_0\in[0,1)$ and each set of $h$ values $E\subset\mathbb{R}$ having zero density at $0$, the quantity $$\ds\limsup_{h\to0^+,h\notin E}\dfrac{f(x_0+h)-f(x_0)}{h}>0.$$  Then $f$ is strictly increasing.
\end{thm}

\begin{proof}
Given any points $0\leq x_1<x_2\leq1$, if $f(x_1)\geq f(x_2)$, then applying O'Malley's~Theorem to the function $f$ restricted to $[x_1,x_2]$, we obtain that $f$ attains an approximate maximum (relative to $[x_1,x_2]$) at some point $x_0\in~[x_1,x_2)$.  That is, there is some set of $h$ values $E_0\subset\mathbb{R}$ having zero density at $0$ such that, on some neighborhood of $x_0$ in $[x_1,x_2]\setminus E_0$, $f$ attains its absolute maximum at $x_0$.

We conclude therefore that $$\ds\limsup_{h\to0^+,h\notin E_0}\dfrac{f(x_0+h)-f(x_0)}{h}\leq0.$$  This contradicts the assumption made on $f$.  We conclude that $f(x_1)<f(x_2)$, and that therefore $f$ is strictly increasing on $[0,1]$.

\end{proof}

\section{The Rest of the Story}
All sets and functions considered here will be assumed to be measurable with respect to $\lambda$, Lebesgue measure on $\mathbb{R}$. Suppose $E\subset \mathbb{R}$ and $I\subset \mathbb{R}$ is an interval. Then the density of $E$ in $I$ is defined as $\Delta(E,I)=\lambda(E\cap I)/\lambda(I).$ Now, if $x\in\mathbb{R}$, then the upper density of $E$ at $x$ is defined as  $\overline\Delta(E,x)=\limsup_{r \to 0} \Delta(E,(x-r,x+r))$. The lower density at $x$, $\underline\Delta(E,x)$ is defined similarly where $\liminf$ replaces $\limsup$ and if these two are equal at $x$, their common value is called the density of $E$ at $x$ and is denoted $\Delta(E,x)$.

Now suppose a function $f:\mathbb{R}\to\mathbb{R}$ is given. Then $f$ is approximately continuous at $x_0$ if there is a set $E$ 
with zero density at $x_0$, so that the limit of $f$ as $x\to x_0$ on $\mathbb{R}\setminus E$ is $f(x_0)$. A function $f$ has an approximate maximum at $x_0$ if $\Delta(H_{f(x_0)},x_0)=0$ where we define $H_y\equiv H_y(f)=\{x:f(x)>y\}$. 

For the purpose of exposition, we isolate the following remarks concerning density and approximately continuous functions. 

\begin{rem}
Although the definitions above of upper and lower density at a point $x$ are given in terms of intervals which are symmetric around $x$ only, it is easy to show that if $E$ and $F$ are sets having density zero and one at $x$ respectively, then for any $\epsilon>0$ there is a $\delta>0$ small enough so that if $I$ is any interval containing $x$ in its closure with $\lambda(I)<\delta$, then $\Delta(E,I)<\epsilon$ and $\Delta(F,I)>1-\epsilon$.
\label{rem: Non-symmetric density.}
\end{rem}

\begin{rem}
Suppose $y \in \mathbb{R}$, $f$ is approximately continuous at $x$ and $\{I_n\}$ is a nested sequence of closed intervals so that that $\{x\}= \bigcap_n I_n$ and $\Delta(H_y,I_n)>\eta>0$ for all $n \in \nn$. Then $\overline{\Delta}(H_y,x)>\eta/2$.
\label{ACcon}
\end{rem}

\begin{rem}
If $f$ is approximately continuous at $x$, $z<y$ and $\overline{\Delta}(H_y,x)>0$ then $f(x)\geqslant y$ and $\Delta(H_z,x)=1$.
\label{ACgreater}
\end{rem}

An important first step is to see that there is an approximately continuous function with no relative extrema. This is, perhaps, not particularly surprising, but the fact that this function can be a derivative is a good introduction into the real nature of derivatives. This example can be found in Andy Bruckner's classic introduction, \cite{AMB}.

\begin{theoremb}
There is a bounded approximately continuous derivative which achieves no local maximum 
and no local minimum.
\end{theoremb}

To visit O'Malley's machinery, consider a measurable set $H$ and an interval $I$ so that $\lambda(H\cap I)>0$ and let $\ee>0$ be given. We define
$$\mathcal{J}(H,I,\ee)=\{J: J \subset I\text{ is an open interval with } \Delta(H,J)>\ee\}$$ and we let $G_\ee(H,I)=\bigcup\mathcal{J}(H,I,\ee) $.
O'Malley proves the following lemma, \cite[Lemma 1]{OM} concerning components of 
$G_\ee(H,I)$. 

\begin{lem}\label{L-one}
Suppose $H\subset [0,1]$ is measurable and fix $(a_0,b_0) \subset (0,1)$ with $\lambda(H \cap (a_0,b_0))>0$ and let $\ee>0$. Let  $(a_1,b_1)$ be a component of $G_\ee(H,(a_0,b_0))$. Then
\begin{enumerate}
\item $\Delta(H,(a_1,b_1))\geqslant \ee/2$, and
\item If $I \subset (a_0,b_0)$ is an open interval with ${I}\cap\{a_1,b_1\}\not=\emptyset$ then  $\Delta(H,I)\leqslant \ee$.
\end{enumerate}
In particular, $\lambda(G_\ee(H,(a_0,b_0)))\leqslant 2 \lambda(H\cap (a_0,b_0))/\ee$. 
\end{lem}

The first item holds since each component is comprised of intervals with density at least $\ee$. The second part of the lemma follows from the definition of $G_\ee(H,(a_0,b_0))$. If $I\subset (a_0,b_0)$ is an interval that either contains or overlaps a component of $G_\ee(H,(a_0,b_0))$ then, by definition, $I \not\in J(H,(a_0,b_0),\ee)$ and so $\Delta(H,I)\leqslant \ee$.

The next lemma has been extracted from the proof of Theorem~1$^*$ in~\cite{OM} and is the main workhorse of the proof our main result. We postpone the proof to Section~\ref{sect: Proof of O'Malley's Lemma.}.

\begin{lemmaom}
Suppose $f:[a,b] \to \mathbb{R}$ is approximately continuous satisfying $\lambda(f^{-1}(y))=0$ for all $y\in\mathbb{R}$.  Let $[a_0,b_0] \subset [a,b]$ be such that $s_0:=\sup f[a_0,b_0]>\max\{f(a_0),f(b_0)\}$. Then for each $\ee>0$ and $y_0 <s_0$ there is a $y_1\in(y_0,s_0)$ and a component $(a_1,b_1)$ of $G_\ee(H_{y_1},(a_0,b_0))$ satisfying:

\begin{enumerate}
\item $[a_1,b_1] \subset (a_0,b_0)$
\item $(b_1-a_1)<1/2(b_0-a_0)$
\item $\max\{f(a_0),f(b_0),y_0\} < y_1$
\item $\max\{f(a_1),f(b_1)\}\leqslant y_1$
\item $\Delta(H_{y_0},(a_1,b_1))>1/2$
\item If $I\subset(a_0,b_0)$ is an open interval with ${I}\cap\{a_1,b_1\}\neq\emptyset$, then $\Delta(H_{y_1},I)\leqslant\ee$.

\end{enumerate}
\end{lemmaom}

Using the above we give the proof of O'Malley's Theorem \ref{mom}$^*$.

\begin{proof}[Proof of O'Malley's Theorem \ref{mom}$^*$]
We first note that if the condition $\lambda(f^{-1}(y))=0$ found in the statement of O'Malley's Lemma fails for some $y\in\mathbb{R}$, then the set $f^{-1}(y)$ will have a density point $x_0\in[0,1)$.  It is easy to see that $f$ will achieve an approximate maximum at $x_0$, and we are done.  Therefore we assume that $\lambda(f^{-1}(y))=0$ for all $y\in\mathbb{R}$.  Note also that as $f$ is not strictly increasing,  there is a $b\in(0,1]$ so that $f(b)\not=\sup f[0,b]$. Therefore without loss of generality we assume that $f(1) \not= \sup f[0,1]$. If $f(0)=\sup f[0,1]$ we are done, setting $x_0=0$.  Therefore we assume further that $\sup f[0,1]>\max\{f(0),f(1)\}$.  Define $[a_0,b_0]=[0,1]$, $s_0=\sup f[0,1]$, and let $y_0<s_0$ be chosen arbitrarily.  We apply O'Malley's Lemma iteratively with $\epsilon=1/k$ at the $k^{\text{th}}$ stage to obtain a strictly increasing sequence of real numbers $\{y_k\}$ and a strictly nested sequence of intervals $\{(a_k,b_k)\}$ such that for each $k\in\mathbb{N}$, $(a_{k+1},b_{k+1})$ is a component of $G_\ee(H_{y_{k+1}},(a_k,b_k))$ and the following items hold.

\begin{enumerate}
\item[(i)] $[a_{k+1},b_{k+1}] \subset (a_k,b_k)$
\item[(ii)] $(b_{k+1}-a_{k+1})<1/2(b_k-a_k)$
\item[(iii)] $\max\{f(a_k),f(b_k),y_k\} < y_{k+1}$
\item[(iv)] $\max\{f(a_{k+1}),f(b_{k+1})\}\leqslant y_{k+1}$
\item[(v)] $\Delta(H_{y_k},(a_{k+1},b_{k+1}))>1/2$
\item[(vi)] If $I\subset(a_k,b_k)$ is an open interval with ${I}\cap\{a_{k+1},b_{k+1}\}\neq\emptyset$, then $\Delta(H_{y_{k+1}},I)\leqslant1/k$.
\end{enumerate}

In order to justify this recursive construction (i.e. to ensure that at the $k^{\text{th}}$ stage the function $f$ and the interval $(a_k,b_k)$ satisfies the assumption made on $f$ and $(a_0,b_0)$ in the statement of the O'Malley's Lemma), observe that since $(a_{k+1},b_{k+1})$ is a component of $G_\ee(H_{y_{k+1}},(a_k,b_k))$, $H_{y_{k+1}}\cap(a_k,b_k)\neq\emptyset$.  Therefore setting $s_k=\sup f[a_k,b_k]$, we have $s_k>y_{k+1}>\max\{f(a_k),f(b_k)\}$.

Items~(i) and~(ii) above yield that the intersection of the intervals so obtained consists of a single point, $\{x_0\} = \bigcap_{k=0}^\infty (a_k,b_k)$.  We have two claims: $f(x_0) \geqslant y_n$ for each $n \in \mathbb{N}$ and $\Delta(H_{f(x_0)},x_0)=0$. This final claim yields that $f$ has an approximate maximum at $x_0\in[0,1)$, as desired. 

Fix some positive integer $n$.  Then for each positive integer $k$, $y_n < y_{n+k}$. Using this inequality and item (v) above
$$\Delta(H_{y_n},(a_{n+k+1},b_{n+k+1}))\geqslant \Delta(H_{y_{n+k}},(a_{n+k+1},b_{n+k+1}))>1/2.$$
\noindent Using Remark~\ref{ACcon}, $\overline{\Delta}(H_{y_n},x_0)>1/4$ and so, by Remark~\ref{ACgreater}, $f(x_0) \geqslant y_n$.

It remains to observe that $\Delta(H_{f(x_0)},x_0)=0$. Suppose by way of contradiction that $\overline{\Delta}(H_{f(x_0)},x_0)=\eta>0$.  Choose some $m\in\mathbb{N}$ with $1/m<\eta/2$.  Since $\overline{\Delta}(H_{f(x_0)},x_0)=\eta$, we can find some $r>0$ small enough so that the interval $I=(x_0-r,x_0+r)$ satisfies
$$I\subset (a_m,b_m),\text{ and }\Delta(H_{f(x_0)},I)>\eta/2.$$
Choose $k>0$ to be the smallest positive integer such that $I\not\subset(a_{m+k+1},b_{m+k+1})$.  Then $$I\subset(a_{m+k},b_{m+k}),\text{ and }I\cap\{a_{m+k+1},b_{m+k+1}\}\neq\emptyset,$$ so we obtain the contradiction
$$\dfrac{\eta}{2}<\Delta(H_{f(x_0)},I)\leqslant\Delta(H_{y_{m+k}},I)<\dfrac{1}{m+k}\leq\dfrac{1}{m}<\dfrac{\eta}{2},$$
where the second inequality follows from the fact that $f(x_0)>y_{m+k}$, and the third inequality above follows from item~(vi).  This contradicts the choice of $r$.
\end{proof}

Several old standards can be immediately generalized to approximately continuous functions
using the theorem above. For example.

\begin{theoremrolle}
Let $f:[0,1]\to \mathbb{R}$ be approximately continuous and approximately  differentiable on $(0,1)$ 
with $f(0)=f(1)$. Then there is a point $x_0\in (0,1)$ at which $f^\prime_{ap}(x_0)=0$.

\end{theoremrolle}

And hence, also it's immediate consequence.
\begin{theoremmvt}
Let $f:[a,b]\to\mathbb{R}$ be approximately continuous and approximately differentiable on $(a,b)$. 
Then there is a point $x_0\in (0,1)$ at which $f^\prime_{ap}(x_0)=\frac{f(b)-f(a)}{b-a}$.
\end{theoremmvt}

Mean Value Theorems for the approximate derivative are well known and in much greater generality; in fact, O'Malley showed in~\cite{OM2} that $x_0$ can be chosen so that
$f^\prime(x_0)=\frac{f(b)-f(a)}{b-a}$.

\section{Proof of O'Malley's Lemma}\label{sect: Proof of O'Malley's Lemma.}

We need the following easy remark.

\begin{rem}\label{remark: Mini.}
Suppose $f:[a,b]\to\mathbb{R}$ is approximately continuous, so that for all $y\in\mathbb{R}$, $\lambda(f^{-1}(y))=0$. Then, setting $s=\sup(f[a,b])$, $\displaystyle\lim_{y\to s}\lambda(H_y)=0$.  Therefore by Lemma~\ref{L-one}, for such a function $f$ and any $\ee>0$, $\lambda(G_\ee(H_y,(a,b)))\to0$ as $y\to s$.
\end{rem}

We restate O'Malley's Lemma for reference.

\begin{lemmaom}
Suppose $f:[a,b] \to \mathbb{R}$ is approximately continuous satisfying $\lambda(f^{-1}(y))=0$ for all $y\in\mathbb{R}$.  Let $[a_0,b_0] \subset [a,b]$ be such that $s_0:=\sup f[a_0,b_0]>\max\{f(a_0),f(b_0)\}$. Then for each $\ee>0$ and $y_0 <s_0$ there is a $y_1\in(y_0,s_0)$ and a component $(a_1,b_1)$ of $G_\ee(H_{y_1},(a_0,b_0))$ satisfying:

\begin{enumerate}
\item\label{item: Subset.} $[a_1,b_1] \subset (a_0,b_0)$
\item\label{item: Shrinking intervals.} $(b_1-a_1)<1/2(b_0-a_0)$
\item\label{item: r_1 greater.} $\max\{f(a_0),f(b_0),y_0\} < y_1$
\item\label{item: r_1 greater equal.} $\max\{f(a_1),f(b_1)\}\leqslant y_1$
\item\label{item: High density interval.} $\Delta(H_{y_0},(a_1,b_1))>1/2$
\item\label{item: Low density endpoints.} If $I\subset(a_0,b_0)$ is an open interval with ${I}\cap\{a_1,b_1\}\neq\emptyset$, then $\Delta(H_{y_1},I)\leqslant\ee$.

\end{enumerate}
\end{lemmaom}

The reader will note a marked similarity between the proof given of O'Malley's Theorem~1$^*$ and the following proof of O'Malley's~Lemma, both depending on a recursive construction of an increasing sequence of real numbers and a nested sequence of intervals.  The thing to notice is that, in the following proof of O'Malley's~Lemma, the two sequences are chosen so that the corresponding members satisfy the first four items in the statement of O'Malley's~Lemma, and then one real number and interval is chosen which also satisfies items~(5) and~(6) as well.  By contrast, each of the real numbers and corresponding intervals found in the proof of O'Malley's~Theorem~1$^*$ satisfy all six items from the statement of O'Malley's~Lemma.  It was then shown that the intersection of all the intervals consists of a single point, which turns out to be the point we were looking for (at which the function achieves a local approximate maximum).

\begin{proof}
Fix $f,(a,b),(a_0,b_0),s_0,y_0$ and $\ee>0$ as in the hypotheses. Find $\alpha>0$ so that $\max\{f(a_0),f(b_0),y_0\} < \alpha < s_0$. Since $f$ is approximately continuous there is a $\delta>0$ so that if $I$ is an interval in $[a,b]$ with either $a_0$ or $b_0$ as an endpoint, and with $\lambda(I)<\delta$, 
\begin{equation}
\Delta(H_\alpha,I)<\ee/2.
\label{smallish}
\end{equation}

In order to choose our number $y_1$ and interval $(a_1,b_1)$, we will first construct a strictly increasing sequence of real numbers $\{r_k\}$ and nested intervals $\{(c_k,d_k)\}$.  To initialize this construction, we set $r_0=y_0$ and $(c_0,d_0)=(a_0,b_0)$.

By Remark~\ref{remark: Mini.}, we may find some $r_1$ with $\alpha < r_1 < s_0$ so that $\lambda(G_\ee(H_{r_1},(a_0,b_0)))< \min\{\delta,1/2(d_0-c_0)\}$. Let $(c_1,d_1)$ be any component of $G_\ee(H_{r_1},(a_0,b_0))$. Assuming that $(c_1,d_1)$ shares an endpoint with $(c_0,d_0)$ yields the contradiction
$$\ee/2\leqslant \Delta(H_{r_1},(c_1,d_1)) \leqslant \Delta(H_{\alpha},(c_1,d_1)) <\ee/2.$$
The first inequality comes from Lemma \ref{L-one}(1), the second from $r_1>\alpha$, and the third from (\ref{smallish}). Thus $[c_1,d_1]\subset (c_0,d_0)$.

So far (1), (2) and (3) are satisfied for $(c_1,d_1)$. To see (4), suppose, by way of contradiction, that $r_1 <\max\{f(c_1),f(d_1)\}$, and without loss of generality that $r_1< f(c_1)$. The approximate continuity of $f$ at $c_1$ then yields that $\Delta(H_{r_1},c_1)=1$. This, in turn, implies that there is an open interval $I$ contained in $(c_0,d_0)$ and containing $c_1$ with $\Delta(H_{r_1},I) > \ee$. This  contradicts Lemma~\ref{L-one}(2).

At this point we have that the interval $(c_1,d_1)$ and the value $r_1$ satisfy items~(\ref{item: Subset.})-(\ref{item: r_1 greater equal.}) in the statement of the lemma (replacing $(a_1,b_1)$ with $(c_1,d_1)$ and $y_1$ with $r_1$).  The next step is to iterate this construction, obtaining a strictly nested sequence of intervals $(c_k,d_k)$ and a strictly increasing sequence of numbers $r_1<r_2<\cdots<s_0$ such that, for each $k\geqslant 1$, $(c_{k+1},d_{k+1})$ is a component of $G_\ee(H_{r_{k+1}},(c_k,d_k))$ satisfying

\begin{enumerate}
\item[(i)] $[c_{k+1},d_{k+1}]\subset(c_k,d_k)$
\item[(ii)] $(d_{k+1}-c_{k+1})<1/2(d_k-c_k)$.
\item[(iii)] $\max\{f(c_k),f(d_k),r_k\}< r_{k+1}$
\item[(iv)] $\max\{f(c_{k+1}),f(d_{k+1})\}\leqslant r_{k+1}$
\end{enumerate}

In order to justify this recursive construction (i.e. to ensure that at the $k^{\text{th}}$ stage the function $f$ and the interval $(c_k,d_k)$ satisfies the assumption made on $f$ and $(a_0,b_0)$ in the statement of the lemma), observe that since $(c_{k+1},d_{k+1})$ is a component of $G_\ee(H_{r_{k+1}},(c_k,d_k))$, $H_{r_{k+1}}\cap(c_k,d_k)\neq\emptyset$.  Therefore setting $s_k=\sup f[c_k,d_k]$, we have $s_k>r_{k+1}>\max\{f(c_k),f(d_k)\}$.

Items~(i) and~(iv) above guarantee that the intersection of the intervals $[c_k,d_k]$ consists of a single point, $\{x_0\} = \bigcap_k[c_k,d_k]$. 

For each $k\geqslant1$, $r_1< r_{k+1}$, so $\Delta(H_{r_1},(c_k,d_k)) \geqslant \Delta(H_{r_{k+1}},(c_k,d_k)) \geqslant \ee/2$ (by Lemma~\ref{L-one}(1)). Therefore using Remark~\ref{ACcon} we obtain $\overline{\Delta}(H_{r_1},x_0) \geqslant \ee/4$. Remark~\ref{ACgreater} now yields that $f(x_0)\geqslant r_1$ and since $r_1>y_0$, $\Delta(H_{y_0},x_0)=1$. Therefore using Remark~\ref{rem: Non-symmetric density.}, we can find some $n$ so that $\Delta(H_{y_0},(c_n,d_n))> 1/2$. Define $y_1=r_n$ and $(a_1,b_1)=(c_n,d_n)$. 

For this choice of $(a_1,b_1)$ and $y_1$ it is easy to verify items~(\ref{item: Subset.})--(\ref{item: r_1 greater equal.}) from the statement of the lemma.  Item~(\ref{item: High density interval.}) follows immediately from the choice of $n$.

It remains to establish item~(\ref{item: Low density endpoints.}). To this end, let $(c,d)\subset(a_0,b_0)$ be any open interval with $(c,d)\cap\{a_1,b_1\}\neq\emptyset$. Now, $(c,d)\subset(c_0,d_0)=(a_0,b_0)$, and $(c,d)\not\subset(c_n,d_n)=(a_1,b_1)$, so we may choose the least $m\in\{1,\ldots,n\}$ such that $(c,d)\not\subset(c_m,d_m)$ (and thus $(c,d)\subset(c_{m-1},d_{m-1})$).  Moreover, $(c,d)$ intersects $(a_1,b_1)$, which is in turn contained in $(c_m,d_m)$, so we have that $(c,d)\cap\{c_m,d_m\}\neq\emptyset$.  Therefore, since $(c_m,d_m)$ is a component of $G_\ee(H_{r_m},(c_{m-1},d_{m-1}))$, it follows from Lemma~\ref{L-one}(2) that $\Delta(H_{r_m},(c,d))\leq\ee$.  Since $y_1\geqslant r_m$, it follows that $\Delta(H_{y_1},(c,d))\leqslant\ee$ as required.

\end{proof}

\section{An example}\label{sect: Ornstein's example.}

In order to show that item (b) in Ornstein's Theorem may not be significantly weakened without losing the result of the theorem, in~\cite{O} a continuous function $f:[0,1]\to\mathbb{R}$ is described which satisfies the following weaker assumption (b'), but which is not monotonic on $[0,1]$.

(b') For each point $x_0\in[0,1]$, the set $E=\left\{x:\frac{f(x)-f(x_0)}{x-x_0}\geq0\right\}$ does not have zero density at $x_0$.

In this section we identify a problem with this example and provide a fix.  
In the last section we show that the typical continuous function satisfies condition~(b')
but is not monotonically increasing.

\subsection{The function and the problem.}

Begin by choosing the eight points, $p_i=i/7$ for $i=0,1,\dots,7$ in $[0,1]$ (the points $\{p_i\}$ have been explicitely chosen for the sake of concreteness, but with any other choice of these points the same problem would occur) and defining a function $g$ at each of those points as follows:

\begin{align}\label{e0}
g(p_0)&=1,\hspace{5pt}g(p_1)=\frac{4}{3},\hspace{5pt}g(p_2)=\frac{1}{3},\hspace{5pt}g(p_3)=\frac{4}{3}, \notag \\
g(p_4)&=\frac{-1}{3},\hspace{5pt}g(p_5)=\frac{2}{3},\hspace{5pt}g(p_6)=\frac{-1}{3},\hspace{5pt}g(p_7)=0.
\end{align}

Extend $g$ linearly on the intervening intervals and let
\begin{equation}\label{den1}
E_{x_0}=\left\{x\in[0,1]:\dfrac{g(x)-g(x_0)}{x-x_0}\geq0\right\}.
\end{equation}
Then $\phi(x)=\Delta(E_{x},[0,1])$
is continuous and, by inspection, positive at each point $x\in[0,1]$. Hence, there 
an $\alpha>0$ such that $\phi(x)>\alpha$ for each $x\in[0,1]$.

A sequence of functions, $g_n$ is defined inductively by first setting $g_0=g$. 

\medskip
\begin{quote}\label{insert}
Assuming $g_n$ has been defined, $g_{n+1}$ is obtained by replacing each decreasing segment of $g_n$ 
with a suitable affine copy of $g$. \hspace*{23mm} ($\star$)
\end{quote}
We refer to the process described in $(\star)$ as the {\em insertion of $g$ into $g_n$.} Specifically, this entails that if $[a,b]$ denotes a maximal interval on which $g_n$ is decreasing, define 
$g_{n+1}(x)=S\circ g\circ T(x)$ for each $x\in [a,b]$ where

\begin{equation}\label{e1}
T(x)=\frac{x-a}{b-a}\text{ and }S(x)=xg(a)+(1-x)g(b).
\end{equation}

The claim is that the sequence $g_n$ converges pointwise to a continuous function, and that this limit function $f$ satisfies the condition (b'). 

Unfortunately, $\{g_n\}$ does not converge to a continuous function. To see this, we will show that there is a sequence of points $\{x_n\}\in[0,1]$ such that for each $n$ the sequence $\{g_k(x_n)\}$ is eventually constant, and these constants approach $\infty$ as $n\to\infty$.  By the compactness of $[0,1]$ this implies that the functions $g_n$ do not converge to a continuous function.

To that end, we define a nested sequence $I_1\supset I_2\supset I_3\cdots$ of sub-intervals of $[0,1]$ inductively as follows.  We set $I_1=[3/7,4/7]$.  Suppose that $I_k$ has been defined for all $k<n$, and $I_{n-1}=[a_{n-1},b_{n-1}]$.  Then we define
$$I_n=\left[a_{n-1}+\frac{3}{7^n},a_{n-1}+\frac{4}{7^n}\right].$$

The construction of $\{g_n\}$ above immediately implies that for each $n$, $I_n$ is a maximal interval of decrease for $g_n$.  Let us also define a sequence $\{\Delta y_n\}$ by $\Delta y_n=g_n(b_n)-g_n(a_n)$, the net change in $g_n$ on $I_n$.  Since $g(1)-g(0)=1$, and when $g_{n}$ is defined on $I_{n-1}$, the vertical scaling factor used is $|\Delta y_{n-1}|$, an easy induction argument shows that

$$\Delta y_n=\Delta y_{0}\cdot|\Delta y_{n-1}|=-\left|\Delta y_0\right|^{n+1}=-\left(\frac{5}{3}\right)^{n+1}.
$$

Using the same reasoning, it may also be shown that

$$g_n(a_n)=1+\displaystyle\sum_{i=0}^n\frac{1}{3}\cdot\left(\frac{5}{3}\right)^i.$$

Of course this sequence $\{g_n(a_n)\}\to\infty$, and we note that since $I_n$ is a maximal interval of decrease of $g_n$, for any $m>n$, $g_m(a_n)=g_n(a_n)$.  Therefore putting $x_n=a_n$, the sequence has the properties described above.  We conclude therefore that the sequence $g_n$ does not converge to a continuous function.

In actuality and with a little more computation it's not hard to see that $\{g_n(\frac{1}{2})\}\to +\infty$.  The purpose of this note, however is simply to point out that the example needs repair and in the following subsection we show how this can be done in a rather straightforward manner.

\section{A fix}

In this section we adapt Ornstein's function so that the change in $y$ values (as in the $\Delta y_n$ from our discussion of Ornstein's functions) on each decreasing interval is strictly less than $1$.  We will use $h$'s here, rather than $g$'s to avoid confusion.

Define $h$ to be the continuous function defined $[0,13]$ with the following prescribed values at the integers, and linear in the intervening intervals.

\begin{center}
\begin{tabular}{ccccccccc}
$h(0)$&=&4/4&$h(1)$&=&6/4&$h(2)$&=&3/4\\
$h(3)$&=&5/4&$h(4)$&=&2/4&$h(5)$&=&4/4\\
$h(6)$&=&1/4&$h(7)$&=&3/4&$h(8)$&=&0/4\\
$h(9)$&=&2/4&$h(10)$&=&-1/4&$h(11)$&=&1/4\\
$h(12)$&=&-2/4&$h(13)$&=&0/4&
\end{tabular}
\end{center}

The function $h$ has been chosen so that at each point $x_0$, $h_0$ takes a smaller value to the left of $x_0$ or a larger value to the right of $x_0$, thus again assuring that the difference quotient is positive on a density (in $[0,13]$) $\alpha$ set (for some $\alpha>0$ independent of $x_0$).  If we choose $\alpha>0$ a bit smaller, we can say a bit more, that for each $x_0\in[0,13]$, the set

\[
E_{x_0}=\left\{x\in[0,13]:\dfrac{h_0(x)-h_0(x_0)}{x-x_0}\geq0\right\},
\]

where we disregard all $x$-values on which $h_0$ is decreasing, has density $\alpha$ in $[0,13]$.  This will be of use to us when we show that our final function $h_\infty$ satisfies the property (b').

Let $h_n$ be the sequence of functions with domain $[0,13]$ defined recursively in 
precisely the same manner as Ornstein's function was defined, but using $h=h_0$ as 
our ``seed function'', rather than Ornstein's $g$.  
That is, 
\begin{quote}
$\dots$ to get $h_{n+1}$ we simply replace each line segment of the graph of $g_n$
having negative slope with an affine copy of $h$. 
\end{quote}
This process 
We will now show that our sequence $h_n$ does converge uniformly on $[0,13]$.  If $h_0$ is decreasing on an interval $[a,b]\subset[0,13]$, then

\[
h_0(b)-h_0(a)\geq-1\cdot\dfrac{3}{4}.
\]

Recursively, if $h_n$ is decreasing on an interval $[a,b]\subset[0,13]$, then

\[
h_n(b)-h_n(a)\geq-1\cdot\left(\dfrac{3}{4}\right)^n.
\]

Therefore the difference between $h_{n+1}$ and $h_n$ on $[a,b]$ is at most $(3/4)^n$ times the difference between $h_0$ and the line $y=-1/13x+1$ on the interval $[0,13]$.  That is, the sequence $\{h_n\}$ converges uniformly.

\subsection{Show that $h_\infty$ satisfies (b').}

If $x_0\in[0,13]$ is contained in one of the intervals on which some one of the $h_n$'s is increasing, then the desired result holds immediately, because the function values of all later $h_{n+k}$'s will not change on that interval.

Suppose that $x_0$ is not in any such interval.  Define $I_0=[0,13]$, and for each $n>0$, let $I_n\subset[0,13]$ denote the interval on which $h_{n-1}$ is decreasing which contains $x_0$.  (That is, $I_n$ is the interval containing $x_0$ on which $h_{n-1}$ is changed to form $h_n$.)  It is easy to see that $m(I_n)=\left(\dfrac{1}{13}\right)^{n-1}$.

Moreover, if we set $$E_n=\left\{x\in I_n:\dfrac{h_\infty(x)-h_\infty(x_0)}{x-x_0}\geq0\right\},$$ (again including only the $x$ values at which $h_n$ is increasing) then we will show that $\Delta(E_n,I_n)\geq1/26$.  Since the intervals $I_n\to\{x_0\}$, this will immediately imply that (b') holds.

We first show that, if $I_1=[1,2]$, then $$\Delta(E_0,I_0)=\Delta(E_0,[0,13])\geq1/26.$$  The idea is that, regardless of the value of $h_\infty(x_0)$, the union $[0,.5]\cup[2.5,3]$ (ie the left half of the interval of increase to the left of $I_1$ and the right half of the interval of  increase to the right of $I_1$) contains at least mass $1/2$ of $E_0$.  Put $y_0=h_\infty(x_0)$.  We proceed by cases.

\begin{case}
$y_0\geq1.25$.
\end{case}

By inspection, $E_0$ contains the interval $[0,0.5]$.  Thus $\Delta(E_0,[0,13])\geq1/26$.

\begin{case}
$y_0\leq1$.
\end{case}

By inspection, $E_0$ contains the interval $[2.5,3]$.  Thus $\Delta(E_0,[0,13])\geq1/26$.

\begin{case}
$1\leq y_0\leq1.25$.
\end{case}

This is the interesting case.  Consider the intervals of the graph with $x$-values $[0,.5]$ and $[2.5,3]$.  $h_\infty$ has constant slope on these intervals, and $$h_\infty(0)=h_\infty(2.5)=1\text{, and }h_\infty(.5)=h_\infty(3)=1.25.$$  Therefore if we choose $\epsilon\in(0,1/2)$ so that $h_\infty(\epsilon)=y_0$, then $E_0$ contains the union $[0,\epsilon]\cap[2.5+\epsilon,3]$.

But $m([0,\epsilon]\cap[2.5+\epsilon,3])=.5$, so we conclude that $\Delta(E_0,[0,13])\geq1/26$.

This argument extends immediately to show that $\Delta(E_0,[0,13])\geq1/26$ for every other choice of $I_1$, always finding points in $E_0$ with mass at least $1/2$ in the intervals of increase of $h_0$ directly to the left and right of $I_1$.

Let us now consider the second interval $I_1$.  We wish to show that $\Delta(E_1,I_1)\geq1/26$.  The first (ie from left to right) possibility for $I_2$ is the interval $[14/13,15/13]$.

Here the argument is the same with the $y$ values for the cases being: Case 1) $y_0\geq27/16$, Case 2) $y_0\leq1.5$, and Case 3) $1.5\leq y_0\leq27/16$, and the intervals where we will be finding points in $E_1$ are $[26/26,27/26]$ (the first half of the interval of increase of $h_1$ to the left of $I_2$) and $[31/26,32/26]$ (the second half of the interval of increase of $h_1$ to the right of $I_2$).  The conclusion is that $m(E_1)$ is greater than or equal to the length of one of these intervals, $m(E_1)\geq1/26$, and $m(I_1)=1$.  Thus $\Delta(E_1,I_1)\geq1/26$.  This iterates nicely (we always pick up a factor of $1/13$ in both the numerator and the denominator of our density calculation, which cancel), thus $\Delta(E_n,I_n)\geq1/26$, and thus the upper density of the set

\[
E=\left\{x\in[0,13]:\dfrac{h_\infty(x)-h_\infty(x_0)}{x-x_0}\geq0\right\}
\]

at $x_0$ is greater than or equal to $1/26>0$, proving that $h_\infty$ satisfies (b').

\section{Counterexamples are Typical}

The goal of this final section is to show that continuous functions with Property (b') 
are ubiquitous in the complete  space $\c$ of all continuous functions on [0,1] endowed with the $\sup$ metric. However, ``ubiquitous'' can be defined in several ways. 

A property is {\em typical} in a complete metric space of functions, $\c$ if the set of functions  
enjoying that property is residual (the complement of a set of the first Baire 
Category) in $\c$.  A well known method of establishing 
whether a given set $A$ is residual or not is the so-called Banach-Mazur Game which we describe 
briefly here. See \cite{Z} for more details and generalizations.

This is a two player game and the players take turns selecting open balls from $\c$. 
Suppose $A\subset\c$ is fixed. Player $P_1$ selects a ball, $B_1$, then player two, $P_2$ selects a ball 
$B_2\subset B_1$ and so on so that the game produces a nested sequence of balls, 
$\{B_n:n\in\N\}$. $P_2$ wins the game if $A\cap\bigcap_{n=1}^\infty B_n \not=\emptyset$ 
otherwise $P_1$ wins. And $P_2$ has a winning strategy provided $P_2$ can always win the game, meaning independently of the balls $P_1$ selects.
\begin{theorembm}
$P_2$ has a winning strategy iff $A$ is residual.
\end{theorembm}

The Banach-Mazur Game is a convenient way to see why the next result is true. The proof uses the
following  notation. If $p_i=(x_i,y_i)\in\R^2,\ i=1,2$ we define
\[
DQ(p_1,p_2)=\frac{y_2-y_1}{x_2-x_1}.
\]
\begin{theorem}
The typical continuous function has Property (b').
\end{theorem}
\begin{proof}

Suppose $P_1$  has chosen the ball $B_n\equiv B_\epsilon(f)\subset\c$ at the $n^{th}$ stage of play.  
We describe the strategy for $P_2$. 
\begin{enumerate}
\item First partition $[0,1]$ into sufficiently small intervals such that the insertion (see
$(\star)$ on page \pageref{insert}) of $h$  into any partition interval lies within the $\epsilon/2$ ball 
about $f$. Let $g:[0,1]\to\R$ be the conjunction of all such insertions and define
\[
E_{x_0}=\left\{x: \frac{g(x)-g(x_0)}{x-x_0}>0\right\}.
\]
Then for every partition interval $J$ and every $x_0\in J$, 
$\Delta(E_{x_0},J)>\alpha$. This function $g$ is the center of the ball $P_2$ will respond with.
\item To determine the response radius, first fix a partition interval $J$ and an $x_0\in J$. There exists  
$\eta(x_0)>0$ such that $\Delta(E_{x_0}\backslash B_{\eta(x_0)},J)>\alpha$. Hence by compactness 
there is a $\eta>0$ such that for every $x\in [0,1]$ and every partition interval $J$ containing $x$,
\[
\Delta(E_{x}\backslash B_\eta(x),J)>\alpha.
\]
Hence, again by compactness, there is a $0<\delta<\epsilon/2$ such that $\delta<\eta$ and if $x_0\in[0,1]$ and
$x_1\in E_{x_0}\backslash B_{\eta}$ then whenever 
\begin{equation}\label{nest}
p_0\in B_\delta((x_0,g(x_0))\text{ and }
p_1\in B_\delta((x_1,g(x_1))\text{ then }DQ(p_0,p_1)>0.
\end{equation}
\end{enumerate}

\smallskip
\noindent Player $P_2$ returns the ball $B_\delta(g)$
where $\delta$ is the radius just determined above.

\medskip
Now, any sequence of plays converges uniformly in the sense that if $f_n$ is any choice of a function in $B_n$, then the sequence of functions $\{f_n\}$ converges uniformly. Due to 
(\ref{nest}), at each $x\in[0,1]$ the density of the set of points for which the difference quotient is 
positive at $x$ exceeds $\alpha$ at the scale of each play of $P_2$. That is, the limit function 
satisfies (b'); this then, completes the proof.
\end{proof}

As it is well known that the set of monotone functions is nowhere dense in $\c$, the following  is an immediate corollary.

\begin{cor}
The typical continuous function satisfies property~(b') but is not monotonically increasing on any interval.
\end{cor}


\begin{thebibliography}{100}
\bibitem{AMB} 
A. M. Bruckner, {\em Differentiation of real functions.}
Second edition.
CRM Monograph Series, 5.
American Mathematical Society, Providence, RI, 1994.
\bibitem{ACLP}
A.~Giovanni, M.~Cs\H ornyei, M.~Laczkovich, D.~Preiss, {\em Denjoy--Young--Saks theorem for approximate derivatives revisited}, {Real Anal. Exchange} {\bf 26 (1)}, 
(2001), 485--488.
\bibitem{GN}
C.~Goffman, C.~J.~Neugebauer, {\em On approximate derivates}, Proc. Amer. Math. {\bf 11(6)}, (1960), 962--966.
\bibitem{OM}
R. J. O'Malley, {\em Approximate maxima}, {Fund. Math.} {\bf 94} (1977), no. 1, 75--81.
\bibitem{OM2}
R. J. O'Malley, {\em The set where an approximate derivative is a derivative},
{Proc. Amer. Math. Soc.} {\bf 54} (1976), 122--124.
\bibitem{O}
D. Ornstein, {\em A characterization of monotone functions},  Illinois J. Math. {\bf 15}(1971), 73--76.
\bibitem{Z}
M. Zeleny, {\em The Banach-Mazur game and $\sigma$-porosity}, Fund. Math. {\bf 150}(1996), 197--210.
\end{thebibliography}
\end{document}